\documentclass[12pt,reqno]{amsart}

\usepackage{amssymb}

\newtheorem{theorem}{Theorem}[section]
\newtheorem{proposition}[theorem]{Proposition}

\theoremstyle{definition}

\numberwithin{equation}{section}

\begin{document}

\title[Torsors over moduli spaces of vector bundles]{Torsors over
moduli spaces of vector bundles over curves of fixed determinant}

\author[I. Biswas]{Indranil Biswas}

\address{Department of Mathematics, Shiv Nadar University, NH91, Tehsil
Dadri, Greater Noida, Uttar Pradesh 201314, India}

\email{indranil.biswas@snu.edu.in, indranil29@gmail.com}

\author[J. Hurtubise]{Jacques Hurtubise}

\address{Department of Mathematics, McGill University, Burnside
Hall, 805 Sherbrooke St. W., Montreal, Que. H3A 2K6, Canada}

\email{jacques.hurtubise@mcgill.ca}



\date{}

\begin{abstract}
Let ${\mathcal M}$ be a moduli space of stable vector bundles of rank $r$ and
determinant $\xi$ on a compact Riemann surface $X$. Fix a semistable holomorphic vector bundle $F$ on
$X$ such that $\chi(E\otimes F)\,=\, 0$ for $E\, \in\, \mathcal M$. Then any $E\,\in\, \mathcal M$ with
$H^0(X,\, E\otimes F)\,=\, 0\,=\, H^1(X,\, E\otimes F)$ has a natural holomorphic projective connection.
The moduli space of pairs $(E,\, \nabla)$, where $E\, \in\, \mathcal M$ and $\nabla$ is a holomorphic
projective connection on $E$, is an algebraic $T^*{\mathcal M}$--torsor on $\mathcal M$. We identify
this $T^*{\mathcal M}$--torsor on $\mathcal M$ with the $T^*{\mathcal M}$--torsor given by the sheaf
of connections on an ample line bundle over $\mathcal M$.
\end{abstract}

\maketitle

\section{Introduction}\label{sec0}

Let $X$ be a compact connected Riemann surface of genus $g$. The holomorphic cotangent bundle of
$X$ is denoted by $K_X$. A holomorphic projective connection
on a holomorphic vector bundle $E$ over $X$ is a holomorphic connection on the projective bundle
${\mathbb P}(E)\, \longrightarrow\, X$. It is same as a holomorphic connection on the corresponding
principal $\text{PGL}(r,{\mathbb C})$--bundle on $X$, where $r$ is the rank of $E$.

We prove the following (see Proposition \ref{prop1}):

\begin{proposition}\label{pr1}
For any holomorphic vector bundle $V$ on $X$ such that
$$
H^0(X,\, V)\ =\ 0\ =\ H^1(X,\, V),
$$
the projective bundle ${\mathbb P}(V)$ has a canonical holomorphic connection.
\end{proposition}

Assume that $g\, \geq\, 2$. Let $\mathcal M$ denote the moduli space of stable vector bundles
$E$ on $X$ of rank $r$ with $\bigwedge^r E\,=\, \xi$, where $\xi$ is a fixed holomorphic line bundle
on $X$. Fix a semistable holomorphic vector bundle $F$ on
$X$ such that $\chi(E\otimes F)\,=\, 0$ for $E\, \in\, \mathcal M$. Let ${\mathcal U}\, \subset\,
\mathcal M$ be the Zariski open subset parametrizing all $E\, \in\, \mathcal M$ such that
$H^0(X,\, E\otimes F)\,=\, 0\,=\, H^1(X,\, E\otimes F)$.

We prove the following (see Proposition \ref{prop2}):

\begin{proposition}\label{pr2}
For any vector bundle $E\, \longrightarrow \, X$ lying in the above open subset $\mathcal U$,
the projective bundle ${\mathbb P}(E)$ has a canonical holomorphic connection.
\end{proposition}

The connection in Proposition \ref{pr2} is constructed using the connection in Proposition \ref{pr1}.

Using $F$ we can construct a line bundle on $\mathcal M$. This line bundle, which will be denoted by $L$,
has a Hermitian connection $\nabla^{L,Q}$ given by a construction of Quillen. Let
${\rm Conn}(L)\, \longrightarrow\, \mathcal M$ denote the fiber bundle given by the space of locally defined
holomorphic splittings of the Atiyah exact sequence for $L$. So the space of holomorphic (respectively, $C^\infty$)
sections of ${\rm Conn}(L)$ over an open subset $U\, \subset\, \mathcal M$ is the space of holomorphic connections
(respectively, complex connections) on the restriction $L\big\vert_U$; see Section \ref{ses} for the construction
of ${\rm Conn}(L)$. Since the space of all holomorphic connections on $L\big\vert_U$ is an affine space for
$H^0(U,\, T^*U)$, the projection ${\rm Conn}(L)\, \longrightarrow\, \mathcal M$ makes ${\rm Conn}(L)$ an
algebraic torsor over $\mathcal M$ for the holomorphic cotangent bundle $T^*\mathcal M$.

Let ${\mathcal N}_C$ denote the moduli space of all pairs of the form $(E,\, \nabla)$, where $E\, \in\,
\mathcal M$ and $\nabla$ is a holomorphic connection on the projective bundle ${\mathbb P}(E)$. The
space of all holomorphic connections on ${\mathbb P}(E)$ is an affine space for
$H^0(X,\, \text{ad}(E)\otimes K_X)\,=\, T^*_E{\mathcal M}$, where $\text{ad}(E)\, \subset\, \text{End}(E)$
is the subbundle of co-rank one given by the sheaf of endomorphisms of trace zero. Therefore, the projection
$$
\Phi\, :\, {\mathcal N}_C \, \longrightarrow\, {\mathcal M},\ \ \ (E,\, {\mathcal D})\, \longmapsto\, E
$$
makes ${\mathcal N}_C$ an algebraic torsor over $\mathcal M$ for the holomorphic cotangent bundle $T^*\mathcal M$.

We produce an isomorphism between the two $T^*\mathcal M$--torsors ${\rm Conn}(L)$ and ${\mathcal N}_C$; see
Theorem \ref{thm1}. The paper closes with some comments on why this correspondence between seemingly unrelated objects should hold.

When $F$ is a theta characteristic, these results were proved earlier in \cite{BH}, \cite{BB}. Indeed the situation considered there, for bundles 
$E$ of degree zero, is to consider the open set $U\, \subset\, \mathcal M$ for which $H^0(X,\, E\otimes K_X^{1/2} )\ =\ 0\ =\ H^1(X,\, E\otimes 
K_X^{1/2} )$. The construction in this paper provided there a connection on $E$, and so a projective connection on $ E\otimes K_X^{1/2}$. Our 
construction here extends this to arbitrary degree, and shows that stability is not required. (See \cite{Der} for other interesting results.)

\section{A canonical projective connection}

Let $X$ be a compact connected Riemann surface. The holomorphic cotangent bundle of $X$ will be
denoted by $K_X$. Take a holomorphic vector bundle $V$ of rank $r$ over $X$. Let
\begin{equation}\label{a1}
p_0\ :\ {\mathbb P}(V)\ \longrightarrow\ X
\end{equation}
be the corresponding projective bundle that parametrizes the hyperplanes in the fibers of
$V$. The space of holomorphic isomorphisms from ${\mathbb C}{\mathbb P}^{r-1}$ to the fibers
of ${\mathbb P}(V)$ produce a holomorphic principal $\text{PGL}(r,{\mathbb C})$--bundle
\begin{equation}\label{a2}
\mathbf{P} \ =\ P_{\text{PGL}(r,{\mathbb C})} \ \longrightarrow \ X.
\end{equation}

A holomorphic connection on the fiber bundle ${\mathbb P}(V)$ in \eqref{a1} is a holomorphic
splitting of the differential
$$
dp_0 \ :\ T{\mathbb P}(V) \ \longrightarrow\ p^*_0 TX
$$
of the projection $p_0$ in \eqref{a1}. There is a natural bijection between the holomorphic connections
on ${\mathbb P}(V)$ and the holomorphic connections on the principal $\text{PGL}(r,{\mathbb C})$--bundle
$\mathbf P$ in \eqref{a2}. A holomorphic connection on ${\mathbb P}(V)$ is also called a holomorphic
projective connection on $V$.

\begin{proposition}\label{prop1}
For any holomorphic vector bundle $V$ on $X$ such that
\begin{equation}\label{a2-a}
H^0(X,\, V)\ =\ 0\ =\ H^1(X,\, V),
\end{equation}
the projective bundle ${\mathbb P}(V)$ has a canonical holomorphic connection.
\end{proposition}

\begin{proof}
For $i\,=\, 1,\, 2$, let $p_i\,:\, X\times X\, \longrightarrow\, X$ be natural the projection
to the $i$-th factor. Let
$$
\Delta\ := \ \{(x,\, x)\,\, \big\vert\,\, x\, \in\, X\}\ \subset \ X\times X
$$
be the reduced diagonal. We will identify $X$ with $\Delta$ via the map $x\, \longmapsto\,
(x,\, x)$. The Poincar\'e adjunction formula says that the restriction of the holomorphic line bundle
${\mathcal O}_{X\times X}(\Delta)$ to $\Delta$ is identified with the normal
bundle $\Delta$, so ${\mathcal O}_{X\times X}(\Delta)\big\vert_\Delta \,=\, TX$ (invoking the identification of
$X$ with $\Delta$). Therefore, we have
\begin{equation}\label{a2c}
\left((p^*_2 K_X)\otimes {\mathcal O}_{X\times X}(\Delta)\right)\big\vert_{\Delta}\,=\,
K_X\otimes TX \,=\, {\mathcal O}_X
\end{equation}
invoking the identification of $X$ with $\Delta$. From \eqref{a2c} it follows that
\begin{equation}\label{a2b}
\left((p^*_1 V)\otimes p^*_2 (V^*\otimes K_X)\otimes {\mathcal O}_{X\times X}(
\Delta)\right)\big\vert_{\Delta}\, = \, \left((p^*_1 V)\otimes (p^*_2 V^*)\right)\big\vert_{\Delta}
\, =\, \text{End}(V).
\end{equation}

{}From \eqref{a2b} we have the following short exact sequence of sheaves on $X\times X$:
\begin{equation}\label{a3}
0\, \longrightarrow\, (p^*_1 V)\otimes p^*_2 (V^*\otimes K_X)
\, \longrightarrow\, (p^*_1 V)\otimes p^*_2 (V^*\otimes K_X)\otimes {\mathcal O}_{X\times
X}(\Delta) \, \longrightarrow\, \text{End}(V)\, \longrightarrow\, 0,
\end{equation}
where $\text{End}(V)$ is supported on $\Delta\,=\, X$. Since $H^1(X,\, V)\,=\, 0$
(see \eqref{a2-a}), the Serre duality says that $H^0(X,\, V^*\otimes K_X)\,=\, H^1(X,\, V)^*\,=\, 0$.
similarly, $H^1(X,\, V^*\otimes K_X)\,=\, 0$ because $H^0(X,\, V)\,=\, 0$.
In view of \eqref{a2-a} and the K\"unneth
formula $$H^i(X\times X,\, (p^*_1 V)\otimes p^*_2 (V^*\otimes K_X))\,=\,\bigoplus_{0\leq j \leq i}
H^j(X,\, V)\otimes H^{i-j}(X,\, V^*\otimes K_X)$$ we conclude that
$$
H^k(X\times X,\, (p^*_1 V)\otimes p^*_2 (V^*\otimes K_X))\ =\ 0
$$
for every $0\,\leq\, k\, \leq\, 2$. Consequently, from the long exact sequence of cohomologies
associated to the short exact sequence in \eqref{a3} the following isomorphism is obtained
\begin{equation}\label{a4}
\eta \ :\ H^0(X\times X,\, (p^*_1 V)\otimes p^*_2 (V^*\otimes K_X)\otimes {\mathcal O}_{X\times
X}(\Delta)) \ \stackrel{\sim}{\longrightarrow}\ H^0(X,\, \text{End}(V)).
\end{equation}
Let
\begin{equation}\label{a5}
\widehat{\phi} \ :=\, \eta^{-1}({\rm Id}_V) \ \in\ H^0(X\times X,\, (p^*_1 V)\otimes p^*_2 (V^*\otimes K_X)\otimes
{\mathcal O}_{X\times X}(\Delta))
\end{equation}
be the section that corresponds to the identity automorphism of $E$ by the isomorphism in \eqref{a4}.

Consider the nonreduced divisor $2\Delta \, \subset\, X\times X$. Let
\begin{equation}\label{a6}
\phi\, :=\, \widehat{\phi}\big\vert_{2\Delta} \, \in\, H^0(2\Delta,\, ((p^*_1 V)\otimes p^*_2 (V^*\otimes K_X)\otimes
{\mathcal O}_{X\times X}(\Delta))\big\vert_{2\Delta})
\end{equation}
be the restriction to $2\Delta$ of the section $\widehat{\phi}$ in \eqref{a5}. We will show that $\phi$ defines
a holomorphic connection on ${\mathbb P}(V)$.

Take any nonempty open subset $U\, \subsetneq\, X$. Let $$\Delta_U\, =\, \Delta\cap (U\times U)\, \subset\, X\times X$$
be the diagonal on $U\times U$. Consider the nonreduced divisor
$$2\Delta_U\, =\, 2\Delta\cap (U\times U)\, \subset\, X\times X.$$
Choose a holomorphic trivialization
\begin{equation}\label{a7}
\rho\, :\, {\mathcal O}_{2\Delta_U}\, \stackrel{\sim}{\longrightarrow}\, ((p^*_2K_X)\otimes{\mathcal O}_{X\times X}
(\Delta))\big\vert_{2\Delta_U}
\end{equation}
satisfying the condition that the restriction of $\rho$ to the reduced divisor $\Delta_U\, \subset\, 2\Delta_U$
coincides with the canonical trivialization in \eqref{a2c}. Let
\begin{equation}\label{a7a}
\phi'_U \, \in\, H^0(2\Delta_U,\, ((p^*_1 V)\otimes (p^*_2 V^*))\big\vert_{2\Delta_U})
\end{equation}
be the unique section such that
\begin{equation}\label{a8}
\phi'_U\otimes \rho(1_{2\Delta_U})\ = \phi\big\vert_{2\Delta_U},
\end{equation}
where $\phi$ is the section in \eqref{a6} and $1_{2\Delta_U}$ is the constant function $1$ on $2\Delta_U$
while $\rho$ is the homomorphism in \eqref{a7}. Since $\eta(\widehat{\phi})\,=\, {\rm Id}_V$ (see \eqref{a5}),
and the restriction of $\rho$ to $\Delta_U\, \subset\, 2\Delta_U$ coincides with the
canonical trivialization in \eqref{a2c}, it follows immediately from \eqref{a8} that the restriction
of $\phi'_U$ to $\Delta_U\, \subset\, 2\Delta_U$ coincides with ${\rm Id}_V\big\vert_U$ (invoking the natural
identification of $U$ with $\Delta_U$). Consequently, $\phi'_U$ defines a holomorphic connection on
the vector bundle $V\big\vert_U$ over $U$ \cite[p.~6, 2.2.4]{De}.

Note that $\left((p^*_2 K_X)\otimes {\mathcal O}_{X\times X}(\Delta)
\otimes {\mathcal O}_{X\times X}(-\Delta)\right)\big\vert_{\Delta}\,=\, K_X$ (using the identification of
$\Delta$ with $X$). So we have the following short exact sequence of sheaves on $2\Delta$:
$$
0\, \longrightarrow\, K_X\, \longrightarrow\, 
\left((p^*_2 K_X)\otimes {\mathcal O}_{X\times X}(\Delta)\right)\big\vert_{2\Delta}
\, \longrightarrow\, \left((p^*_2 K_X)\otimes {\mathcal O}_{X\times X}(\Delta)\right)\big\vert_{\Delta}
\, =\, {\mathcal O}_X \, \longrightarrow\, 0,
$$
where both $K_X$ and ${\mathcal O}_X$ are supported on $\Delta$. From this exact sequence it follows that
if we replace $\rho$ in \eqref{a7} with another such trivialization $\widetilde{\rho}$
of $\left((p^*_2 K_X)\otimes {\mathcal O}_{X\times X}(\Delta)\right)\big\vert_{2\Delta}$ with the same
property that the restriction of $\widetilde\rho$ to $\Delta_U\, \subset\, 2\Delta_U$ coincides with the
canonical trivialization of $\left((p^*_2 K_X)\otimes {\mathcal O}_{X\times X}(\Delta)\right)\big\vert_{\Delta}$
in \eqref{a2c}, then
\begin{equation}\label{c1}
\widetilde{\rho}- \rho \ \in \ H^0(U,\, K_X\big\vert_U).
\end{equation}
Let
$$
\widetilde{\phi}'_U \, \in\, H^0(2\Delta_U,\, ((p^*_1 V)\otimes (p^*_2 V^*))\big\vert_{2\Delta_U})
$$
be the section constructed as in \eqref{a7a} by substituting $\widetilde{\rho}$ in place of $\rho$. Then from
\eqref{a8} it follows that
$$
{\phi}'_U - \widetilde{\phi}'_U
\ =\ (\widetilde{\rho}- \rho)\otimes {\rm Id}_{V}\big\vert_U\ \in \ H^0(U,\, K_X\big\vert_U)\otimes {\rm Id}_{V}\big\vert_U
$$
(see \eqref{c1}). So the holomorphic connection on $V\big\vert_U$ given by $\widetilde{\phi}'_U$ differs from the
holomorphic connection on $V\big\vert_U$ given by ${\phi}'_U$ by a section of $(K_X\otimes{\rm Id}_V)\big\vert_U$.
Hence these two holomorphic connections on $V\big\vert_U$ produce the same holomorphic connection on ${\mathbb P}
(V)\big\vert_U$. Thus $\phi$ in \eqref{a6} produces a holomorphic connection on ${\mathbb P}(V)$ over $X$.
\end{proof}

\section{Projective connections on stable vector bundles}

\subsection{Projective unitary connections}

Let $X$ be a compact connected Riemann surface of genus $g$, with $g\, \geq\,2$.
Fix a holomorphic line bundle $\xi$ on $X$. Fix an integer $r \, \geq\, 2$. If
$g\,=\,2$, then assume that $r \, \geq\, 3$.

Let
\begin{equation}\label{b1}
{\mathcal M}\ =\ {\mathcal M}(r,\xi)
\end{equation}
be the moduli space of stable vector bundles $E$ on $X$ with ${\rm rank}(E)\,=\, r$ and
$\bigwedge^r E \,=\,\xi$. It is a smooth complex quasiprojective variety of dimension
$(r^2-1)(g-1)$. For any $E\,\in\, {\mathcal M}$, the projective bundle ${\mathbb P}(E)$ admits a 
holomorphic connection \cite{We}, \cite{At}, \cite{AzBi}.

Let ${\mathcal N}_C$ denote the moduli space of all pairs of the form $(E,\, {\mathcal D})$, where
$E\, \in\, {\mathcal M}$ (defined in \eqref{b1}) and ${\mathcal D}$ is a holomorphic connection
on the projective bundle ${\mathbb P}(E)$. Let
\begin{equation}\label{b2}
\Phi\, :\, {\mathcal N}_C \, \longrightarrow\, {\mathcal M},\ \ \ (E,\, {\mathcal D})\, \longmapsto\, E
\end{equation}
be the natural surjective projection. For any $E\, \in\, {\mathcal M}$, the space of all holomorphic
connections on ${\mathbb P}(E)$ is an affine space for $H^0(X,\, \text{ad}(E)\otimes K_X)$, where
${\rm ad}(E)\, \subset\, \text{End}(E)$ is the subbundle of co-rank one given by the sheaf of endomorphisms
of $E$ of trace zero. On the other hand, we have $T^*_E {\mathcal M} \,=\, H^0(X,\, \text{ad}(E)\otimes K_X)$.
The map $\Phi$ in \eqref{b2} gives an algebraic torsor on ${\mathcal M}$ for the cotangent bundle
$T^*{\mathcal M}$.

A theorem of Narasimhan and Seshadri says that any $E\, \in\, {\mathcal M}$ has the property that
the projective bundle ${\mathbb P}(E)$ has a unique holomorphic connection for which the image
of the monodromy homomorphism $\pi_1(X,\, x_0) \, \longrightarrow\, \text{PGL}(E_{x_0})$
lies in a maximal compact subgroup of $\text{PGL}(E_{x_0})$ \cite{NS}; here $x_0\, \in\, X$ is any
base point. Therefore, there is a unique section of the map $\Phi$ in \eqref{b2} 
$$
\beta \,:\, {\mathcal M}\, \longrightarrow\, {\mathcal N}_C,
$$
meaning
\begin{equation}\label{b3}
\Phi\circ\beta\ =\ {\mathcal M},
\end{equation}
satisfying the condition that for any $E\, \in\, {\mathcal M}$ the image $\beta(E)\,=\, (E, {\mathcal D})$
has the property that ${\mathcal D}$ is the holomorphic connection on ${\mathbb P}(E)$ for which the image
of the monodromy homomorphism $\pi_1(X,\, x_0) \, \longrightarrow\, \text{PGL}(E_{x_0})$
lies in a maximal compact subgroup of $\text{PGL}(E_{x_0})$. The map $\beta$ in \eqref{b3} is $C^\infty$ but it is not holomorphic.

\subsection{A canonical connection on stable bundles}

A theorem of Faltings says that a holomorphic vector bundle $V$ on $X$ is semistable if and
only if there is a holomorphic vector bundle $W$ on $X$ such that
$$
H^0(X,\, V\otimes W)\,=\, 0\,=\, H^1(X,\, V\otimes W)
$$
\cite[p.~514, Theorem 1.2]{Fa}. In particular, for every $E\, \in\, {\mathcal M}$ (defined
in \eqref{b1}), there is some
holomorphic vector bundle $F$ (which depends on $E$) satisfying the condition that
\begin{equation}\label{e1}
H^0(X,\, E\otimes F)\,=\, 0\,=\, H^1(X,\, E\otimes F).
\end{equation}
Fix a holomorphic vector bundle
\begin{equation}\label{e0}
F\ \longrightarrow\ X
\end{equation}
satisfying the condition that there is some $E\, \in \, {\mathcal M}$
for which \eqref{e1} holds. Note that \eqref{e1} it follows that $F$ is semistable.
Indeed, if $F$ is not semistable, and $F_1\, \subset\, F$ is the unique maximal semistable 
subbundle of $F$, then $\chi(E\otimes F_1)\, >\, 0$ for all $E\, \in \, {\mathcal M}$.
This implies that $H^0(X,\, E\otimes F_1)\, \not=\, 0$ and hence
$H^0(X,\, E\otimes F)\, \not=\, 0$, which contradicts \eqref{e1}. So $F$ must be semistable.

Using semicontinuity of dimension of cohomologies it follows that
there is a nonempty Zariski open subset ${\mathcal U}\, \subset\, {\mathcal M}$ such that
\eqref{e1} holds for all $V\, \in\, {\mathcal U}$. The complement
\begin{equation}\label{e2}
{\mathcal M}\setminus {\mathcal U} \ \subset\ \mathcal M
\end{equation}
is a divisor. The Picard group of $\mathcal M$ is isomorphic to ${\mathbb Z}$ \cite{DN},
and the line bundle ${\mathcal O}_{\mathcal M}({\mathcal M}\setminus {\mathcal U})$
is a positive multiple of the ample generator of ${\rm Pic}({\mathcal M})\,=\, {\mathbb Z}$.
Indeed, if $F$ is of the minimal rank $r'$ required for a zero Euler characteristic of $E\otimes F$, ($r'
\,=\, rk(E)/GCD(rk(E),\, deg(E))$),
then by \cite{DN}, the theta-divisor ${\mathcal M}\setminus {\mathcal U}$ is a generator of the Picard group.

\begin{proposition}\label{prop2}
For any vector bundle $E\, \longrightarrow \, X$ lying in the open subset $\mathcal U$ in
\eqref{e2}, the projective bundle ${\mathbb P}(E)$ has a canonical holomorphic connection.
\end{proposition}

\begin{proof}
Take any $E\,\in\, \mathcal U$, and consider $E\otimes F$, where $F$ is the vector bundle in
\eqref{e0}. In view of \eqref{e1}, from Proposition \ref{prop1} it follows that
the projective bundle ${\mathbb P}(E\otimes F)$ has a canonical holomorphic connection. Let
\begin{equation}\label{e4}
\nabla
\end{equation}
denote the canonical holomorphic connection on the projective bundle ${\mathbb P}(E\otimes F)$.

Let $r'$ be the rank of $F$. Denote by ${\mathcal F}$ the principal $\text{PGL}(r',
{\mathbb C})$--bundle on $X$ defined by ${\mathbb P}(F)$. Denote by ${\mathcal E}$ the principal
$\text{PGL}(r,{\mathbb C})$--bundle on $X$ defined by ${\mathbb P}(E)$. Denote by
${\mathcal G}$ the principal $\text{PGL}(rr',{\mathbb C})$--bundle on $X$
defined by ${\mathbb P}(E\otimes F)$. The decomposition $E\otimes F$ of the vector
bundle of rank $rr'$ as a tensor product produces a reduction of structure group
\begin{equation}\label{e5}
\iota \ :\ {\mathcal E}\times_X {\mathcal F} \ \hookrightarrow\ {\mathcal G}
\end{equation}
of the $\text{PGL}(rr',{\mathbb C})$--bundle $\mathcal G$ to the subgroup
\begin{equation}\label{e5b}
\text{PGL}(r,{\mathbb C})\times \text{PGL}(r', {\mathbb C})\, \hookrightarrow\,
\text{PGL}(rr', {\mathbb C});
\end{equation}
the natural action of $\text{GL}(r,{\mathbb C})\times\text{GL}(r', {\mathbb C})$ on
${\mathbb C}^r\otimes {\mathbb C}^{r'}$ defines a homomorphism
$\text{GL}(r,{\mathbb C})\times \text{GL}(r', {\mathbb C})\, \longrightarrow\,
\text{PGL}(rr', {\mathbb C})$ which in turn produces an injective homomorphism from
$\text{PGL}(r,{\mathbb C})\times \text{PGL}(r', {\mathbb C})$ to
$\text{PGL}(rr', {\mathbb C})$. The connection $\nabla$ in \eqref{e4} produces a
$sl(rr', {\mathbb C})$--valued holomorphic $1$--form on the total space of $\mathcal G$ (recall
that $sl(rr', {\mathbb C})$ is the Lie algebra of $\text{PGL}(rr',{\mathbb C})$).
This $sl(rr', {\mathbb C})$--valued holomorphic $1$--form on $\mathcal G$ will be
denoted by $\vartheta$. So
\begin{equation}\label{e6}
\iota^*\vartheta \ \in\ H^0({\mathcal E}\times_X {\mathcal F},\,
T^*({\mathcal E}\times_X {\mathcal F})\otimes_{\mathbb C} sl(rr', {\mathbb C}))
\ =\ H^0({\mathcal E}\times_X {\mathcal F},\,
T^*({\mathcal E}\times_X {\mathcal F}))\otimes sl(rr', {\mathbb C})
\end{equation}
is a $sl(rr', {\mathbb C})$--valued holomorphic $1$--form on the total space of ${\mathcal E}\times_X
{\mathcal F}$, where $\iota$ is the map in \eqref{e5}.

The homomorphism of Lie algebras for the homomorphism of Lie groups in \eqref{e5b}
is an injective homomorphism of $\text{PGL}(r,{\mathbb C})\times\text{PGL}(r', {\mathbb C})$--modules
$$
sl(r, {\mathbb C})\oplus sl(r', {\mathbb C}) \ \longrightarrow\ sl(rr', {\mathbb C}).
$$
Pre-composing it with the natural inclusion map $sl(r, {\mathbb C})\,\hookrightarrow\,
sl(r, {\mathbb C})\oplus sl(r', {\mathbb C})$ defined by $v\,\longmapsto\, (v,\, 0)$ we get an
injective homomorphism of $\text{PGL}(r,{\mathbb C})\times\text{PGL}(r', {\mathbb C})$--modules
\begin{equation}\label{e7}
I \ :\ sl(r, {\mathbb C}) \ \longrightarrow\ sl(rr', {\mathbb C});
\end{equation}
note that $\text{PGL}(r', {\mathbb C})$ acts trivially on $sl(r, {\mathbb C})$.
We will construct a splitting of this homomorphism $I$ of
$\text{PGL}(r,{\mathbb C})\times\text{PGL}(r', {\mathbb C})$--modules.

We have
$$
sl(rr', {\mathbb C})\oplus {\mathbb C}\ =\ {\rm End}({\mathbb C}^r\otimes {\mathbb C}^{r'}) \ =\ {\rm End}({\mathbb C}^r)\otimes
{\rm End}({\mathbb C}^{r'})
$$
$$
=\ (sl(r, {\mathbb C})\oplus {\mathbb C})\otimes (sl(r', {\mathbb C})\oplus {\mathbb C})
\ =\ sl(r, {\mathbb C})\oplus A,
$$
where $A\,=\, (sl(r, {\mathbb C})\otimes sl(r', {\mathbb C})) \oplus sl(r', {\mathbb C})\oplus {\mathbb C}$.
Let
\begin{equation}\label{e8}
J \ :\ sl(rr', {\mathbb C}) \ \longrightarrow\ sl(r, {\mathbb C})
\end{equation}
be the following composition of homomorphisms:
$$
sl(rr', {\mathbb C}) \, \hookrightarrow\, sl(rr', {\mathbb C})\oplus {\mathbb C}\,=\, sl(r, {\mathbb C})\oplus A\, 
\longrightarrow\, sl(r, {\mathbb C}),
$$
where $sl(r, {\mathbb C})\oplus A\, \longrightarrow\, sl(r, {\mathbb C})$ is the natural projection.

{}From the construction of $J$ in \eqref{e8} it is evident that
\begin{itemize}
\item $J$ is a homomorphism of $\text{PGL}(r,{\mathbb C})\times\text{PGL}(r', {\mathbb C})$--modules, and

\item $J\circ I \ = \ {\rm Id}_{sl(r, {\mathbb C})}$, where $I$ is the homomorphism in \eqref{e7}.
\end{itemize}
Consider the form $\iota^*\vartheta$ constructed in \eqref{e6}. Note that
$J\circ \iota^*\vartheta$ is a $sl(r, {\mathbb C})$--valued holomorphic $1$--form on the total
space of ${\mathcal E}\times_X {\mathcal F}$. Next, consider the natural projection
\begin{equation}\label{ep}
\varpi\ :\ {\mathcal E}\times_X {\mathcal F}\ \longrightarrow\ {\mathcal E}.
\end{equation}
{}From the above two properties of $J$ and the properties of $\vartheta$ it
follows immediately that $J\circ \iota^*\vartheta$ descends to a
$sl(r, {\mathbb C})$--valued holomorphic $1$--form on the total space of ${\mathcal E}$. In other words,
there is a unique $sl(r, {\mathbb C})$--valued holomorphic $1$--form $\widetilde{\vartheta}$ on the total
space of ${\mathcal E}$ such that
\begin{equation}\label{ep2}
\varpi^*\widetilde{\vartheta} \ =\ J\circ \iota^*\vartheta,
\end{equation}
where $\varpi$ is the projection in \eqref{ep}. It is now straightforward to check that the
restriction of $\widetilde{\vartheta}$ to the fibers of $\mathcal E$ coincides with the Maurer--Cartan form,
and $\widetilde{\vartheta}$ is $\text{PGL}(r,{\mathbb C})$--equivariant. Consequently,
$\widetilde{\vartheta}$ defines a holomorphic connection on the principal $\text{PGL}(r,{\mathbb C})$--bundle
$\mathcal E$. This completes the proof.
\end{proof}

Associating to any $E\, \in\, \mathcal U$ (see \eqref{e2})
the holomorphic connection on ${\mathbb P}(E)$ constructed in Proposition
\ref{prop2}, we get a map
\begin{equation}\label{e8a}
\gamma \ :\ {\mathcal U}\ \longrightarrow\ {\mathcal N}_C
\end{equation}
(see \eqref{b2}). The map $\gamma$ in \eqref{e8a} is algebraic, and also
\begin{equation}\label{e8b}
\Phi\circ\gamma\ =\ \text{Id}_{\mathcal U},
\end{equation}
where $\Phi$ is the projection in \eqref{b2}.

It should be clarified that the map $\gamma$ depends on $F$
in \eqref{e0}. In fact, $\mathcal U$ depends on $F$.

The section $\beta$ of $\Phi$ in \eqref{b3} is defined on entire $\mathcal M$, but it is not holomorphic.
In contrast, the section $\gamma$ in \eqref{e8a} is algebraic but it is defined only on $\mathcal U$.

\section{The theta line bundle on moduli space}

\subsection{A line bundle and a section}

By convention, the top exterior product of the $0$--dimensional complex vector space is
$\mathbb C$.

Consider the vector bundle $F$ in \eqref{e0}. We have a natural algebraic line bundle
\begin{equation}\label{el}
L\ \longrightarrow\ {\mathcal M}
\end{equation}
on the moduli space (see \eqref{b1}) whose fiber over any $E\, \in\, {\mathcal M}$
is $(\bigwedge^{\rm top} H^0(X,\, E\otimes F)^*)\otimes (\bigwedge^{\rm top} H^1(X,\, E\otimes F))$.
To construct $L$ explicitly, take an open subset $U$, in \'etale topology, of ${\mathcal M}$ such that there
is a Poincar\'e vector bundle ${\mathcal E}\, \longrightarrow\, X\times U$. Let
\begin{equation}\label{p2}
\psi\, :\, X\times U \, \longrightarrow\, U \ \ \,\text{ and }\, \ \
\eta\, :\, X\times U \, \longrightarrow\, X
\end{equation}
be the natural projections. Consider the line bundle
$$
{\rm Det}'({\mathcal E})\ :=\ (\det \psi_*({\mathcal E}\otimes\eta^*F)^*)\otimes
(\det R^1\psi_*({\mathcal E}\otimes\eta^*F))\ \longrightarrow\ U
$$
(see \cite[Ch.~V, \S~6]{Ko} for the construction of the determinant line bundle). Take another
Poincar\'e line bundle ${\mathcal E}\otimes\psi^*\zeta$ on $X\times U$. Then by the projection formula,
$$
{\rm Det}'({\mathcal E}\otimes\zeta)\otimes \zeta^{\otimes \chi(E\otimes E)}\ =\ 
{\rm Det}'({\mathcal E}),
$$
where $E\, \in\, {\mathcal M}$ and $\chi(E\otimes E)\,=\, \dim H^0(X,\, E\otimes F)-
\dim H^1 (X,\, E\otimes F)$ (it is independent of the choice of $E$). But from
\eqref{e1} it follows that $\chi(E\otimes E)\,=\, 0$. Hence we have
$$
{\rm Det}'({\mathcal E}\otimes\zeta) \ =\ {\rm Det}'({\mathcal E}).
$$
Consequently, these locally defined line bundles on $\mathcal M$ of the form ${\rm Det}'
({\mathcal E})$ patch together compatibly to produce a line bundle $L$ as in \eqref{el}.

The line bundle $L$ in \eqref{el} has a natural section which will be described next.

Fix a point $x_0\, \in\, X$ and an integer $d_0 \, >\, 2(g-1) -\text{degree}(\xi)/r$ (see
\eqref{b1}). Denote the divisor $d_0x_0$ on $X$ by $D_0$. For any $E\, \in\, {\mathcal M}$, we have
$$
H^1(X,\, E\otimes {\mathcal O}_X(D_0)) \,=\, H^0(X,\, E^*\otimes {\mathcal O}_X(-D_0)\otimes K_X)^*
$$
by Serre duality. Now $E^*\otimes {\mathcal O}_X(-D_0)\otimes K_X$ is stable because $E$ is so, and
$$\text{degree}(E^*\otimes {\mathcal O}_X(-D_0)\otimes K_X)
\,=\, 2(g-1)r - d_0r - \text{degree}(\xi) \, <\, 0.$$
Hence $H^0(X,\, E^*\otimes {\mathcal O}_X(-D_0)\otimes K_X)\,=\, 0$, which implies that
\begin{equation}\label{z1}
H^1(X,\, E\otimes {\mathcal O}_X(D_0)) \,=\,0.
\end{equation}
As before, take an open subset $U$, in \'etale topology, of ${\mathcal M}$ such that there
is a Poincar\'e vector bundle ${\mathcal E}\, \longrightarrow\, X\times U$. Consider the following
short exact sequence of sheaves on $X\times U$:
\begin{equation}\label{ex1}
0 \, \longrightarrow\, {\mathcal E} \, \longrightarrow\, {\mathcal E}\otimes \eta^* {\mathcal O}_X(D_0)
\, \longrightarrow\, ({\mathcal E}\otimes \eta^* {\mathcal O}_X(D_0))\big\vert_{D_0\times U}\, \longrightarrow\, 0,
\end{equation}
where $\eta$ is the map in \eqref{p2}. From \eqref{z1} it follows that
$$
R^1\psi_* ({\mathcal E}\otimes \eta^* {\mathcal O}_X(D_0))\ =\ 0,
$$
where $\psi$ is the map in \eqref{p2}. Therefore, taking the direct image of \eqref{ex1} using $\psi$ we have
the exact sequence
\begin{equation}\label{ex2}
0 \, \longrightarrow\, \psi_* {\mathcal E} \, \longrightarrow\, \psi_*({\mathcal E}\otimes \eta^*
{\mathcal O}_X(D_0)) \, \xrightarrow{\,\,\,{\mathcal R}\,\,\,}\, \psi_*(({\mathcal E}\otimes \eta^*
{\mathcal O}_X(D_0))\big\vert_{D_0\times U})\, \longrightarrow\, R^1\psi_* {\mathcal E} \, \longrightarrow\, 0.
\end{equation}
Note that the support of $({\mathcal E}\otimes \eta^* {\mathcal O}_X (D_0))\big\vert_{D_0\times U}$
is finite over $U$ which implies that $R^1\psi_*({\mathcal E}\otimes \eta^* {\mathcal O}_X (D_0))\big\vert_{D_0\times U}$
vanishes. Hence
$\psi_*(({\mathcal E}\otimes \eta^* {\mathcal O}_X (D_0))\big\vert_{D_0\times U})$ is a vector
bundle on $U$. Also, from \eqref{z1} it follows that $\psi_*({\mathcal E}\otimes \eta^*
{\mathcal O}_X(D_0))$ is a vector bundle on $U$. Now from \eqref{ex2} it follows that
\begin{equation}\label{e12}
L\big\vert_U\ :=\ \det (\psi_* {\mathcal E})^*
\otimes \det (R^1 \psi_* {\mathcal E})
\end{equation}
$$
=\ \det (\psi_*({\mathcal E}\otimes \eta^* {\mathcal O}_X(D_0)))^*\otimes 
\det (\psi_*(({\mathcal E}\otimes \eta^* {\mathcal O}_X(D_0))\big\vert_{D_0\times U})).
$$
The homomorphism $\mathcal R$ in \eqref{ex2} produces a homomorphism
$$
\det {\mathcal R}\ :\ \det (\psi_*({\mathcal E}\otimes \eta^*
{\mathcal O}_X(D_0))) \ \longrightarrow \ \det (\psi_*(({\mathcal E}\otimes\eta^*
{\mathcal O}_X(D_0))\big\vert_{D_0\times U})).
$$
In view of \eqref{e12}, this homomorphism $\det {\mathcal R}$ produces a section
$$
\Gamma_U\ \in \ H^0(U,\, L\big\vert_U).
$$
This section $\Gamma_U$ does not depend on the choices of $x_0$ or $d_0$. It is also independent
of the choice of the Poincar\'e bundle $\mathcal E$. Consequently, these locally defined sections
of the form $\Gamma_U$ patch together compatibly to produce a section
\begin{equation}\label{e13}
\Gamma \ \ \in\ \ H^0({\mathcal M},\, L).
\end{equation}
It is straightforward to see that $\Gamma$ does not vanish on $\mathcal U$ (see \eqref{e2}).
Indeed, both the kernel and the cokernel of the homomorphism $\mathcal R$ in \eqref{ex2} vanish over $\mathcal U$.
So $\mathcal R$ is an isomorphism over $\mathcal U$. Hence the homomorphism $\det {\mathcal R}$ does
not vanish over any point of $\mathcal U$.
In fact, the support of the divisor for $\Gamma$ is exactly ${\mathcal M}\setminus {\mathcal U}$.
This is because the kernel and cokernel of the homomorphism $\mathcal R$ are both nonzero over any point
of ${\mathcal M}\setminus {\mathcal U}$. Hence the homomorphism $\det {\mathcal R}$ vanishes on every
point of the complement ${\mathcal M}\setminus {\mathcal U}$.

\subsection{Sheaf of connections}\label{ses}

The Atiyah exact sequence for the line bundle $L$ in \eqref{el} is the following:
$$
0\, \longrightarrow\, \text{Diff}^0(L,\, L)\,=\, {\mathcal O}_{\mathcal M} \, \longrightarrow\, 
\text{Diff}^1(L,\, L) \, \longrightarrow\, T{\mathcal M} \, \longrightarrow\, 0,
$$
where $\text{Diff}^i(L,\, L)$ is the sheaf on $X$ of holomorphic differential operators from $L$ to itself
of degree $i$ and the above projection $\text{Diff}^1(L,\, L) \, \longrightarrow\, T{\mathcal M}$ is the
symbol map. For any open subset $U\, \subset\, {\mathcal M}$, a holomorphic connection on $L\big\vert_U$
is a holomorphic splitting, over $U$, of the Atiyah exact sequence \cite{At}. Let
\begin{equation}\label{e9}
0\, \longrightarrow\, T^*{\mathcal M} \, \longrightarrow\, \text{Diff}^1(L,\, L)^* \,
\stackrel{\sigma}{\longrightarrow}\, {\mathcal O}_{\mathcal M} \, \longrightarrow\, 0
\end{equation}
be the dual of the Atiyah exact sequence. So a holomorphic connection on $L\big\vert_U$ is a holomorphic splitting
of \eqref{e9} over $U$. Denote the constant function $1$ on $\mathcal M$ by $1_{\mathcal M}$. Consider the inverse image
$$
\text{Conn}(L)\ :=\ \sigma^{-1}(1_{\mathcal M}({\mathcal M}))\ \subset\ \text{Diff}^1(L,\, L)^* ,
$$
where $\sigma$ is the projection in \eqref{e9}.
Let
\begin{equation}\label{e10}
\Psi\ :\ \text{Conn}(L)\ \longrightarrow\ {\mathcal M}
\end{equation}
be the natural projection. From \eqref{e9} it follows immediately that $\text{Conn}(L)$ is a torsor over $\mathcal M$
for $T^*\mathcal M$. For any open subset $U\, \subset\, {\mathcal M}$, a holomorphic connection on $L\big\vert_U$ is a holomorphic
section, over $U$, of the projection $\Psi$ in \eqref{e10}.

The section $\Gamma$ in \eqref{e13} produces a trivialization of $L\big\vert_{\mathcal U}$ (recall
that $\Gamma$ does not vanish on any point over $\mathcal U$).
So the connection on ${\mathcal O}_{\mathcal U}$ given by the de Rham differential $d$ produces
a connection on $L\big\vert_{\mathcal U}$ using the isomorphism ${\mathcal O}_{\mathcal U}\,
\longrightarrow\, L\big\vert_{\mathcal U}$ given by the trivialization of $L\big\vert_{\mathcal U}$.
This connection on $L\big\vert_{\mathcal U}$ defines a section, over $\mathcal U$,
\begin{equation}\label{e11}
\delta\ :\ {\mathcal U} \ \longrightarrow\ \text{Conn}(L)\big\vert_{\mathcal U}
\end{equation}
of the projection $\Psi$ in \eqref{e10}.

Next we will describe a $C^\infty$ section of the projection $\Psi$ in \eqref{e10}.
As mentioned before, for every stable vector bundle $E$ on $X$ the projective bundle
${\mathbb P}(E)$ has a unique holomorphic connection for which the image
of the monodromy homomorphism $\pi_1(X,\, x_0) \, \longrightarrow\, \text{PGL}(E_{x_0})$
lies in a maximal compact subgroup of $\text{PGL}(E_{x_0})$ \cite{NS}. Applying Hodge
theory of unitary local systems this unique holomorphic connection produces a K\"ahler
form on the moduli space ${\mathcal M}$ \cite{AtBo}. Let
\begin{equation}\label{e12b}
\omega_{\mathcal M}\ \in\ C^\infty(M;\, \Omega^{1,1}_M)
\end{equation}
denote this K\"ahler form on $\mathcal M$.

Using the above unique holomorphic connection, Quillen constructed a Hermitian structure on the line
bundle $L$ in \eqref{el} \cite{Qu} (see \cite{BGS1}, \cite{BGS2}, \cite{BGS3} for similar
constructions in a more general set-up). Denote this Hermitian structure on $L$ by $h_Q$. The
Hermitian connection on $L$, for the Hermitian structure $h_Q$, will be denoted by $\nabla^{L
,Q}$. So $\nabla^{L,Q}$ is a $C^\infty$ section of the projection $\Psi$ in \eqref{e10}.
Let $c_1(\nabla^{L,Q})\, \in\, C^\infty(M;\, \Omega^{1,1}_M)$ be the corresponding
Chern form. There is a positive real number $\lambda$ satisfying the condition that
\begin{equation}\label{e13b}
c_1(\nabla^{L,Q})\ =\ \lambda\cdot \omega_{\mathcal M},
\end{equation}
where $\omega_{\mathcal M}$ is the K\"ahler form in \eqref{e12b} \cite{Qu} (see \cite{BGS1},
\cite{BGS2}, \cite{BGS3} for more general results).

\begin{theorem}\label{thm1}
There is a unique real number $\lambda_0$ and a unique algebraic isomorphism of
fiber bundles over $\mathcal M$
$$
\mathcal{I}\ :\ {\rm Conn}(L) \ \longrightarrow\ {\mathcal N}_C
$$
(see \eqref{e10} and \eqref{b2}) such that for any $E\, \in\, {\mathcal M}$ and any
$z\, \in\, {\rm Conn}(L)_E\,=\, \Psi^{-1}(E)$,
$$
\mathcal{I}(z +w)\ =\ \mathcal{I}(z) + \lambda_0\cdot w
$$
for all $w\, \in\, T^*_E {\mathcal M}$.
\end{theorem}

\begin{proof}
Let us first show that given $\lambda_0$, there are no two distinct isomorphisms
satisfying the above condition. Let
$$
\mathcal{I}\ :\ \text{Conn}(L) \ \longrightarrow\ {\mathcal N}_C \ \ \text{ and }\ \
\mathcal{J}\ :\ \text{Conn}(L) \ \longrightarrow\ {\mathcal N}_C
$$
be two algebraic isomorphism of fiber bundles over $\mathcal M$ such that
for any $E\, \in\, {\mathcal M}$ and any $z\, \in\, \Psi^{-1}(E)$,
$$
\mathcal{I}(z +w)\, =\, \mathcal{I}(z) + \lambda_0\cdot w \ \ \text{ and }\ \
\mathcal{J}(z +w)\, =\, \mathcal{J}(z) + \lambda_0\cdot w
$$
for all $w\, \in\, T^*_E {\mathcal M}$. Therefore,
\begin{equation}\label{e14}
\mathcal{I}(z +w) - \mathcal{J}(z +w)\ = \ \mathcal{I}(z) - \mathcal{J}(z).
\end{equation}
Let $\mathcal{I} - \mathcal{J}$ be the algebraic $1$-form on $\mathcal M$
that sends any $E\, \in\, {\mathcal M}$ to
$\mathcal{I}(z) - \mathcal{J}(z) \, \in\, T^*_E {\mathcal M}$, where
$z\, \in\, \Psi^{-1}(E)$; note that from \eqref{e14} it follows immediately that
$\mathcal{I}(z) - \mathcal{J}(z) \, \in\, T^*_E {\mathcal M}$ is independent of
the choice of $z$. On the other hand,
\begin{equation}\label{e17}
H^0({\mathcal M},\, T^* {\mathcal M})\ = \ 0,
\end{equation}
because $\mathcal M$ is unirational. This implies that $\mathcal{I}\,=\, \mathcal{J}$.

Now we will show that there is a unique real number $\lambda_0$ satisfying the
condition that there is an algebraic isomorphism of fiber bundles over $\mathcal M$
$$
\mathcal{I}\ :\ \text{Conn}(L) \ \longrightarrow\ {\mathcal N}_C
$$
such that for any $E\, \in\, {\mathcal M}$ and any
$z\, \in\, \text{Conn}(L)_E\,=\, \Psi^{-1}(E)$,
$$
\mathcal{I}(z +w)\ =\ \mathcal{I}(z) + \lambda_0\cdot w
$$
for all $w\, \in\, T^*_E {\mathcal M}$.

To prove this, we first note that the isomorphism classes of algebraic $T^*{\mathcal M}$--torsors on
$\mathcal M$ are parametrized by $H^1({\mathcal M},\, T^*{\mathcal M})$. Given a $T^*{\mathcal M}$--torsor
$\mathbb V$ on ${\mathcal M}$ we may trivialize it on an open open covering of $\mathcal M$. On the
intersection of two open subsets in this covering the difference of two trivializations of $\mathbb V$
produce a section of $T^*{\mathcal M}$. This way we get an element of $H^1({\mathcal M},\, T^*{\mathcal M})$
associated to $\mathbb V$. A construction of the Dolbeault cohomology class
of this element of $H^1({\mathcal M},\, T^*{\mathcal M})$ is as follows. Take a $C^\infty$
section $\sigma$ of $\mathbb V$ over ${\mathcal M}$. Now $\overline{\partial}\sigma$ is a
$\overline{\partial}$--closed $(1,\,1)$-form on ${\mathcal M}$, and hence it gives an element of
$H^1({\mathcal M},\, T^*{\mathcal M})$.

We have
\begin{equation}\label{e15}
H^1({\mathcal M},\, T^*{\mathcal M})\ =\ H^2({\mathcal M},\,{\mathbb C})\ =\ {\mathbb C}.
\end{equation}
Consider the $C^\infty$ section $\nabla^{L,Q}$ in \eqref{e13b} of the
$T^*{\mathcal M}$--torsor $\text{Conn}(L)$ in \eqref{e10}. From \eqref{e13b} it follows
immediately that the $T^*{\mathcal M}$--torsor $\text{Conn}(L)$ in \eqref{e10} is given by a
positive real number.

Next consider the $C^\infty$ section $\beta$ in \eqref{b3} of $\Phi$. The $(1,\,1)$--form
$\overline{\partial}\beta$ on $\mathcal M$ has the property that there is a positive real
number $\nu$ such that
\begin{equation}\label{e16}
\overline{\partial}\beta \ =\ \nu\cdot \omega_{\mathcal M},
\end{equation}
where $\omega_{\mathcal M}$ is the K\"ahler form in \eqref{e12b} \cite{ZT}, \cite{Iv}.
In view of \eqref{e16}, and the above observation that the $T^*{\mathcal M}$--torsor $\text{Conn}(L)$
in \eqref{e10} is given by a positive real number, the theorem follows from \eqref{e15}.
\end{proof}

\begin{proposition}\label{prop3}
The sections $\nabla^{L,Q}$ and $\beta$ of $\Psi$ and $\Phi$ respectively (see \eqref{e13b}
and \eqref{b3}) satisfy the following:
$$
\mathcal{I}\circ \nabla^{L,Q}\ =\ \beta,
$$
where $\mathcal{I}$ is the isomorphism in Theorem \ref{thm1}.
\end{proposition}

\begin{proof}
Since $\mathcal{I}\circ \nabla^{L,Q}$ and $\beta$ are two $C^\infty$ sections of $\Phi$,
$$
\beta - \mathcal{I}\circ \nabla^{L,Q}\ \in\ C^\infty({\mathcal M},\, T^*{\mathcal M}).
$$
Now from the construction of $\mathcal I$ it follows that $\overline{\partial}(\beta - \mathcal{I}
\circ\nabla^{L,Q})\,=\, 0$. Therefore, from \eqref{e17} it follows that $\beta -
\mathcal{I}\circ \nabla^{L,Q}\,=\, 0$.
\end{proof}

The above isomorphism might seem a bit surprising; after all, we are going from a connection defined on a bundle on the
curve, to a connection on the theta-line bundle over the moduli space. It is perhaps less so when one considers that the section in
$$H^0(X\times X,\, (p^*_1 E\otimes F)\otimes p^*_2 (E^*\otimes F\otimes K_X)\otimes {\mathcal O}_{X\times
X}(\Delta))$$
whose restriction to the diagonal gives the identity section in $End(E\otimes F)$ and whose restriction to the first formal neighbourhood defines 
the projective connection, also provides over the whole of $X\times X$ a Cauchy (or Szeg\"o) kernel for $E\otimes F$ \cite{BB}. Thus we have an 
inverse for the $\overline \partial$-operator on $E\otimes F$, and so an isomorphism between the determinant of $\overline \partial$ and the 
determinant of the identity operator, and thus a section of the determinant bundle.

\section*{Acknowledgements}

The first-named author acknowledges the support of a J. C. Bose Fellowship (JBR/2023/000003).

\section*{Statements and Declarations}

There is no conflict of interests regarding this manuscript. No data were generated or used.


\end{document}